\documentclass[a4paper,11pt]{article}
\usepackage{braket,hyperref}
\usepackage{amssymb,mathtools,amsthm,amsmath}
\usepackage{tikz}
\usepackage{xargs}

\usepackage{comment}

\usepackage{enumitem}
\usepackage{verbatim}
\usepackage{authblk}
\usepackage[margin=3cm]{geometry}

\theoremstyle{plain}
\newtheorem{theorem}{\bf Theorem}
\newtheorem*{theorem*}{Theorem}

\newtheorem{conjecture}[theorem]{\bf Conjecture}
\newtheorem{proposition}[theorem]{\bf Proposition}
\newtheorem{corollary}[theorem]{\bf Corollary}
\newtheorem{lemma}[theorem]{\bf Lemma}
\theoremstyle{definition}

\newenvironment{remark}[1][Remark.]{\begin{trivlist}
		\item[\hskip \labelsep {\bfseries #1}]}{\end{trivlist}}

\numberwithin{theorem}{section}
\numberwithin{equation}{section}

\newcommand{\Rea}{{\mathbb R}}

\newcommand{\fhomology}[2]{\tilde{H}_{#1}\left(#2\right)}
\newcommand{\rhomology}[2]{\tilde{H}_{#1}\left(#2 \right)}

\newcommand{\relhomology}[2]{H_{#1}\left(#2\right)}

\DeclareMathOperator{\lk}{lk}
\DeclareMathOperator{\cost}{cost}
\DeclareMathOperator{\st}{st}

\newcommand{\tol}[2]{\mathcal{T}_{#2}(#1)}
	
\begin{document}

	\title{Leray numbers of tolerance complexes}
	
 \author{Minki Kim\thanks{\href{mailto:minkikim@ibs.re.kr}{minkikim@ibs.re.kr}. This work was supported by the Institute for Basic Science (IBS-R029-C1).}}
 \affil{Discrete Mathematics Group, Institute for Basic Science,
Daejeon, South Korea.}
    
    \author{Alan Lew\thanks{\href{mailto:alan@campus.technion.ac.il}{alan@campus.technion.ac.il}. Supported by ISF grant no. 326/16.}}
	\affil{Department of Mathematics, Technion, Haifa 32000, Israel}
	
	\date{\today}
	\maketitle

\begin{abstract}
    Let $K$ be a simplicial complex on vertex set $V$. $K$ is called \emph{$d$-Leray} if the homology groups of any induced subcomplex of $K$ are trivial  in dimensions $d$ and higher. $K$ is called \emph{$d$-collapsible} if it can be reduced to the void complex by sequentially removing a simplex of size at most $d$ that is contained in a unique maximal face.

   We define the \emph{$t$-tolerance complex} of $K$, $\tol{K}{t}$, as the simplicial complex on vertex set $V$ whose simplices are formed as the union of a simplex in $K$ and a set of size at most $t$.
   We prove that for any $d$ and $t$ there exists a positive integer $h(t,d)$ such that, for every $d$-collapsible complex $K$, the $t$-tolerance complex $\mathcal{T}_t(K)$ is $h(t,d)$-Leray. 
   
   The definition of the complex $\tol{K}{t}$ is motivated by results of Montejano and Oliveros on ``tolerant" versions of Helly's theorem. As an application, we present some new tolerant versions of the colorful Helly theorem.
\end{abstract}

\section{Introduction}
Let $\mathcal{F}$ be a finite family of non-empty sets. The {\em nerve} of $\mathcal{F}$ is the simplicial complex
\[
N(\mathcal{F}) = \{\mathcal{F}' \subset \mathcal{F}: \bigcap_{A\in\mathcal{F}'} A \neq \varnothing \}.
\]
A simplicial complex $K$ is called {\em $d$-representable} if it isomorphic to the nerve of a family of convex sets in $\mathbb{R}^d$.

Two closely related notions are the $d$-collapsibility and $d$-Lerayness of a complex, which were introduced by Wegner in \cite{Weg75} as generalizations of $d$-representability. Let $V$ be a finite set, and let $K$ be a simplicial complex on vertex set $V$.
A face $\sigma$ in $K$ is said to be {\em free} if there exists exactly one maximal face of $K$ that contains $\sigma$.
Given a free face $\sigma$ of cardinality at most $d$, we call the operation of removing from $K$ all faces containing $\sigma$ an {\em elementary $d$-collapse}.
The complex $K$ is called {\em $d$-collapsible} if all of its faces can be removed by performing a sequence of elementary $d$-collapses.
The {\em collapsibility number} of $K$, denoted by $C(K)$, is the minimum integer $d$ such that $K$ is $d$-collapsible.

For an integer $i\geq -1$, let $\fhomology{i}{K}$ be the reduced $i$-dimensional homology group of $K$ with coefficients in $\mathbb{R}$.
For $U\subset V$, the subcomplex of $K$ induced by $U$ is the complex $K[U]=\{\sigma\in K:\, \sigma\subset U\}.$ The complex $K$ is called {\em $d$-Leray} if all its induced subcomplexes have trivial reduced homology groups in dimensions $d$ and higher; that is, if
$ \fhomology{i}{K[U]} = 0$ for all $i \geq d$ and $U\subset V$.
The {\em Leray number} of $K$, denoted by $L(K)$, is the minimum integer $d$ such that $K$ is $d$-Leray.

It was shown in \cite{Weg75} that every $d$-representable complex is $d$-collapsible, and that every $d$-collapsible complex is $d$-Leray.
See the survey paper by Tancer~\cite{Tan13} for an overview of $d$-representability, $d$-collapsibility and $d$-Lerayness.

Helly's theorem~\cite{Hel23} is a fundamental result in combinatorial geometry that asserts the following: for every finite family of convex sets in $\mathbb{R}^d$, if every subfamily of size at most $d+1$ has a point in common, then the whole family has a point in common.
See, for example, \cite{ALP17} for an overview of results and open problems related to Helly's theorem.


Let $K$ be a complex on vertex set $V$. A \emph{missing face} of $K$ is a set $\tau\subset V$ such that $\tau\notin K$ but $\sigma\in K$ for any $\sigma\subsetneq \tau$. The \emph{Helly number} of $K$, denoted by \emph{$h(K)$}, is the maximum dimension of a missing face of $K$. Helly's theorem is equivalent to the fact that, if $K$ is $d$-representable, then $h(K)\leq d$.

Since the boundary of a $k$-dimensional simplex has non-trivial homology in dimension $k-1$, then every $d$-Leray complex $K$ does not contain the boundary of a $k$-dimensional simplex as an induced subcomplex, for any $k> d$. This shows that the bound $h(K) \leq d$ also holds when $K$ is $d$-Leray (and therefore the same holds when $K$ is $d$-collapsible).

Let $\mathcal{H}$ be an $r$-uniform hypergraph on vertex set $V$.
The \emph{covering number} of $\mathcal{H}$, denoted by $\tau(\mathcal{H})$, is the minimum size of a set $U\subset V$ such that $U$ intersects all the edges of $\mathcal{H}$. 
The hypergraph $\mathcal{H}$ is called {\em $t$-critical} if $\tau(\mathcal{H})=t$ and $\tau(\mathcal{H}')<t$ for every hypergraph $\mathcal{H}'$ that is obtained from $\mathcal{H}$ be removing an edge. 
The \emph{Erd\H{o}s-Gallai number} $\eta(r,t)$ is the maximum number of vertices in an $r$-uniform $t$-critical hypergraph.
Equivalently, $\eta(r,t)$ is the minimum positive integer $n$ such that for every $r$-uniform hypergraph $\mathcal{H}$ with $\tau(\mathcal{H}) > t$ there exists an $\mathcal{H}' \subset \mathcal{H}$ with $|V(\mathcal{H}')| \leq n$ such that $\tau(\mathcal{H}') > t$.
Erd\H{o}s and Gallai showed in \cite{Erd61} that $\eta(2,t)=2t$ and $\eta(r,2)=\left\lfloor \left(\frac{r+2}{2}\right)^2\right\rfloor$. For general $r$ and $t$, Tuza proved in \cite{Tuz85} the bound
\[
    \eta(r,t)< \binom{r+t-1}{r-1}+\binom{r+t-2}{r-1},
\]
which is tight up to a constant factor. 
In particular, we have $\eta(r,t)=O(t^{r-1})$ for $r$ fixed and $t\to\infty$, and $\eta(r,t)=O(r^t)$ for $t$ fixed and $r\to\infty$.

Let $\mathcal{F}$ be a finite family of sets. 
We say that $\mathcal{F}$ has a \emph{point in common with tolerance $t$} if there is a subfamily $\mathcal{F}'\subset \mathcal{F}$ such that $|\mathcal{F}'|\geq |\mathcal{F}|-t$ and $\cap_{A\in\mathcal{F}'} A \neq \varnothing$. 
In \cite{MO11}, Montejano and Oliveros proved the following ``tolerant version" of  Helly's theorem.
\begin{theorem}[Montejano and Oliveros~{\cite[Theorem 3.1]{MO11}}]\label{thm:tolerance_helly_convex}
Let $\mathcal{F}$ be a finite family of convex sets in $\Rea^d$, and let $t\geq0$ be an integer. If every subfamily $\mathcal{F}'\subset\mathcal{F}$ of size at most $\eta(d+1,t+1)$ has a point in common with tolerance $t$, then $\mathcal{F}$ has a point in common with tolerance $t$.
\end{theorem}

In fact, it was shown in \cite{MO11} that any family of sets satisfying a Helly property satisfies also a corresponding ``tolerant Helly property". In terms of simplicial complexes, this may be stated as follows:

Let $K$ be a simplicial complex on vertex set $V$, and let $t\geq 0$ be an integer. Define the simplicial complex
\begin{align*}
    \tol{K}{t} &=  \{\eta\cup\tau:\, \eta\in K, \, \tau\subset V,\, |\tau|\leq t\}     \\
    &= \{\sigma\subset V: \exists \eta\subset \sigma, |\sigma\setminus \eta|\leq t, \, \eta\in K\}.
\end{align*}
We call $\tol{K}{t}$ the \emph{$t$-tolerance complex} of $K$. Note that $\tol{K}{0}=K$ for every complex $K$.

\begin{theorem}[Montejano-Oliveros {\cite[Theorem 1.1]{MO11}}]\label{thm:tolerance_helly_general}
Let $K$ be a simplicial complex with $h(K)\leq d$, and let $t\geq 0$ be an integer. Then, $h(\tol{K}{t})\leq \eta(d+1,t+1)-1$.
\end{theorem}

It is natural to ask whether we can achieve a stronger conclusion by strengthening the assumptions on $K$. By replacing the Helly number with the collapsibility or Leray number, the following conjectures arise:

\begin{conjecture}\label{conj:tolerance_leray_to_leray}
Let $K$ be a $d$-Leray simplicial complex. Then, $\tol{K}{t}$ is $(\eta(d+1,t+1)-1)$-Leray.
\end{conjecture}

\begin{conjecture}\label{conj:tolerance_coll_to_coll}
Let $K$ be a $d$-collapsible simplicial complex. Then, $\tol{K}{t}$ is $(\eta(d+1,t+1)-1)$-collapsible.
\end{conjecture}

Let $t\geq 1$, and let $A$ and $B$ be two disjoint sets of size $t+1$ each. Let $K$ be the simplicial complex on vertex set $A\cup B$ whose maximal faces are the sets $A$ and $B$. It is easy to check that $K$ is $1$-collapsible, and therefore $1$-Leray (in fact, it is easy to show that it is even $1$-representable). On the other hand, the complex $\tol{K}{t}$ is the boundary of the simplex $A\cup B$. That is, $\tol{K}{t}$ is a $2t$-dimensional sphere. In particular, it is not $2t$-Leray. Therefore, for $d=1$, the bound $\eta(2,t+1)-1= 2t+1$ in Conjectures \ref{conj:tolerance_leray_to_leray} and \ref{conj:tolerance_coll_to_coll} cannot be improved.

For $t=1$, it was shown in \cite[Theorem 3.2]{MO11} that there exists a $d$-representable complex $K$ such that $\tol{K}{1}$ is the boundary of a $\left(\left\lfloor \left(\frac{d+3}{2}\right)^2\right\rfloor-1\right)$-dimensional simplex. In particular, $\tol{K}{1}$ is not $\left(\left\lfloor \left(\frac{d+3}{2}\right)^2\right\rfloor-2\right)$-Leray. Therefore, for $t=1$, the bound $\eta(d+1,2)-1=\left\lfloor \left(\frac{d+3}{2}\right)^2\right\rfloor-1$ in Conjectures \ref{conj:tolerance_leray_to_leray} and \ref{conj:tolerance_coll_to_coll} cannot be improved.

Our main result is the following:

\begin{theorem}\label{thm:main_tolerance}
Let $K$ be a $d$-collapsible complex, and let $t\geq 0$ be an integer. Then, $\tol{K}{t}$ is $h(t,d)$-Leray, where $h(0,d)=d$ for all $d\geq 0$, and for $t>0$,
\[
    h(t,d)= \left(\sum_{s=1}^{\min\{t,d\}} \binom{d}{s}(h(t-s,d)+1)\right) + d.
\]
\end{theorem}

Note that we require the stronger property (collapsibility) for $K$, and obtain only the weaker property (Leray) for the tolerance complex. For $d=1$, we obtain the tight bound $h(t,1)=2t+1=\eta(2,t+1)-1$.  For $d>1$, $h(t,d)$ is larger than the conjectural bound $\eta(d+1,t+1)-1$. However, when $t$ is fixed, we have $h(t,d)=O(d^{t+1})$, which is of the same order of magnitude as that of $\eta(d+1,t+1)-1$.

In the special case $d=2, t=1$, we can prove the following stronger bound:
\begin{theorem}\label{thm:main_tolerance_d2t1case}
Let $K$ be a $2$-collapsible complex. Then, $\tol{K}{1}$ is $5$-Leray.
\end{theorem}
Note that, since $5=\eta(3,2)-1$, the bound in Theorem \ref{thm:main_tolerance_d2t1case} is tight. 

\subsection{Organization}
The paper is organized as follows. In Section~\ref{sec:background} we present some basic facts about simplicial complexes, homology and collapsibility. In Section~\ref{sec:tolerance_homological} we prove some auxiliary topological results that we will use later. In Section~\ref{sec:tolerance_main} we prove our main result, Theorem \ref{thm:main_tolerance}. In Section~\ref{sec:tolerance_d2t1} we prove Theorem~\ref{thm:main_tolerance_d2t1case} about the Leray number of the $1$-tolerance complex of a $2$-collapsible complex. In Section \ref{sec:colorful} we present some applications to tolerant versions of the colorful Helly theorem.

\section{Background}\label{sec:background}

In this section we recall some basic definitions and results about simplicial complexes, homology and collapsibility. 

\subsection{Simplicial complexes}

Let $V$ be a finite set and let $K\subset 2^V$ be a family of sets. $K$ is called a \emph{simplicial complex} if $\sigma\in K$ for all $\tau\in K$ and $\sigma\subset \tau$. The set $V$ is called the \emph{vertex set} of $K$. A set $\sigma\in K$ is called a {\em simplex} or a {\em face} of $K$. The \emph{dimension} of a simplex $\sigma\in K$ is $\dim(\sigma)=|\sigma|-1$. For short, we call a $k$-dimensional simplex a \emph{$k$-simplex}. Let $K(k)$ be the set of all $k$-simplices.

The \emph{dimension} of the complex $K$, denoted by $\dim(K)$, is the maximal dimension of a simplex in $K$.

$K'$ is a \emph{subcomplex} of $K$ if it is a simplicial complex, and each simplex of $K'$ is also a simplex of $K$.

Let $U\subset V$. The subcomplex of $K$ \emph{induced} by $U$ is the complex
\[
K[U]=\{\sigma\in K:\, \sigma\subset U\}.
\]
Let $\tau\in K$. We define the \emph{link} of $\tau$ in $K$ to be the subcomplex
\[
\lk(K,\tau)= \{\sigma\in K: \, \sigma\cap\tau=\varnothing,\, \sigma\cup\tau \in K\},
\] 
the \emph{star} of $\tau$ in $K$ to be the subcomplex
\[
\st(X,\tau)= \{\sigma\in K:\, \sigma\cup\tau\in K\}
\]
and the \emph{costar} of $\tau$ in $K$ to be the subcomplex
\[
\cost(K,\tau) = \{\sigma\in K:\, \tau \not\subset\sigma\}.
\]

If $\tau=\{v\}$, we write $\lk(K,v)=\lk(K,\{v\})$, $\st(K,v)=\st(K,\{v\})$ and $K\setminus v= \cost(K,\{v\})=K[V\setminus \{v\}]$.

Let $X,Y$ be simplicial complexes on disjoint vertex sets. 
We define the \emph{join} of $X$ and $Y$ to be the simplicial complex
\[
X \ast Y= \{ \sigma\cup \tau:\,  \sigma\in X, \tau\in Y\}.
\]

Let $v\in V$. If $v\in \tau$ for every maximal face $\tau\in K$ we say that $K$ is a \emph{cone over $v$}.

For $U\subset V$, we denote by $2^U=\{\sigma:\, \sigma\subset U\}$ the \emph{complete complex} on vertex set $U$.

\subsection{Simplicial homology}

Let $K$ be a simplicial complex. Let $\fhomology{k}{K}$ be the $k$-th reduced homology group of $K$ with coefficients in $\Rea$.
We say that $K$ is \emph{acyclic} if $\fhomology{k}{K}=0$ for all $k\geq -1$.

A useful tool for computing homology is the Mayer-Vietoris long exact sequence:

\begin{theorem}[Mayer-Vietoris]\label{thm:mayer_vietoris}
Let $X,Y$ be simplicial complexes. Then, the following sequence is exact
\[
\cdots\to \rhomology{k}{X\cap Y} \to \rhomology{k}{X}\bigoplus\rhomology{k}{Y} \to \rhomology{k}{X\cup Y}\to \rhomology{k-1}{X\cap Y}\to \cdots
\]
\end{theorem}

The following special case will be of use later:

\begin{theorem}\label{thm:exact_sequence_link_costar}
Let $K$ be a simplicial complex on vertex set $V$, and let $v\in V$. Then, the following sequence is exact
\[
\cdots\to \rhomology{k}{\lk(K,v)} \to \rhomology{k}{K\setminus v}\to \rhomology{k}{K}\to \rhomology{k-1}{\lk(K,v)}\to \cdots
\]
\end{theorem}
\begin{proof}
Let $A=K\setminus v$ and $B=\st(K,v)$.
By Theorem \ref{thm:mayer_vietoris}, we have a long exact sequence
\[
\cdots\to \rhomology{k}{A\cap B} \to \rhomology{k}{A}\bigoplus\rhomology{k}{B} \to \rhomology{k}{A\cup B}\to \rhomology{k-1}{A\cap B}\to \cdots
\]
Note that $B$ is a cone over $v$, and therefore $\rhomology{k}{B}=0$ for all $k$. Moreover, $A\cup B=K$ and $A\cap B=\lk(K,v)$.  So, we obtain a long exact sequence
\[
\cdots\to \rhomology{k}{\lk(K,v)} \to \rhomology{k}{K\setminus v}\to \rhomology{k}{K}\to \rhomology{k-1}{\lk(K,v)}\to \cdots
\]
\end{proof}

Another useful way of computing homology is by the application of nerve theorems.
Let $X_1,\ldots,X_m$ be simplicial complexes. The \emph{nerve} of the family $\{X_1,\ldots,X_m\}$ is the simplicial complex
\[
N(\{X_1,\ldots,X_m\})= \left\{ I\subset[m]:\, \bigcap_{i\in I} X_i\neq \{\varnothing\}\right\},
\]
where $[m]:=\{1,2,\ldots,m\}$. Roughly speaking, given a family of simplicial complexes, a nerve theorem describes how much the topology of the nerve of the family reflects the topology of the union of the complexes, when every non-empty intersection of the complexes satisfies certain topological restrictions (see e.g. \cite[Theorem 6]{Bjor03} or \cite[Theorem 2.1]{Mes01}).
Here, we will use the following simple version of the nerve theorem:
\begin{theorem}[Leray's Nerve Theorem]\label{thm:nerve_theorem}
Let $X_1,\ldots,X_m$ be simplicial complexes, and let $X=\bigcup_{i=1}^m X_i$. If, for every $I\subset [m]$, $\bigcap_{i\in I} X_i$ is either empty or acyclic, then
\[
	\fhomology{k}{X}\cong \fhomology{k}{N(\{X_1,\ldots,X_m\})}
\]
for all $k\geq -1$.
\end{theorem}

For the union of simplicial complexes, the Leray number can be bounded by the following result by Kalai and Meshulam.

\begin{theorem}[Kalai and Meshulam~\cite{KM06}]\label{thm:kalai_meshulam_leray}
Let $X = \bigcup_{i=1}^{r} X_i$.
Then, \[L(X) \leq \sum_{i=1}^{r}(L(X_i)+1)-1.\]
\end{theorem}

\subsubsection{Relative homology}

Let $X$ be a simplicial complex and let $Y$ be a subcomplex of $X$. Let $C_k(X,Y)$ be the $\Rea$-vector space generated by the ordered $k$-simplices in $X\setminus Y$, under the relations
\[
[v_0,\ldots,v_k]=\text{sgn}(\pi) [v_{\pi(0)},\ldots,v_{\pi(k)}],
\]
for every $k$-simplex $\{v_0,\ldots,v_k\}\in X\setminus Y$ and permutation $\pi:\{0,\ldots,k\}\to \{0,\ldots,k\}$.
We define a linear map $\partial_k: C_k(X,Y)\to C_{k-1}(X,Y)$ that acts on the spanning set by
\[
    \partial_k[v_0,\ldots,v_k]= \sum_{\substack{i\in\{0,\ldots,k\}:\\ \{v_0,\ldots,v_{i-1},v_{i+1},\ldots,v_k\}\notin Y}} (-1)^i [v_0,\ldots,v_{i-1},v_{i+1},\ldots,v_k].
\]
We define the group of \emph{$k$-cycles} as $Z_k(X,Y)=\text{Ker}(\partial_k)$ and the group of \emph{$k$-boundaries} as $B_k(X,Y)=\text{Im}(\partial_{k+1})$. For any $k$, we have $B_k(X,Y)\subset Z_k(X,Y)$, so we can define the quotient
\[
\relhomology{k}{X,Y}=Z_k(X,Y)/B_k(X,Y).
\]
We call $\relhomology{k}{X,Y}$ the $k$-th relative homology group of the pair $Y \subset X$.
The relative homology of the pair $Y\subset X$ is related to the homology of the two complexes via the following result:

\begin{theorem}[Long exact sequence of a pair]\label{thm:exact_sequence_pair}
Let $Y\subset X$ be simplicial complexes. Then, the following sequence is exact:
\[
\cdots\to\fhomology{k}{Y}\to \fhomology{k}{X}\to \relhomology{k}{X,Y}\to \fhomology{k-1}{Y}\to\cdots
\]
\end{theorem}

\subsection{Collapsibility}

We will need the following properties, showing that $d$-collapsibility is  ``hereditary":

\begin{lemma}[Wegner \cite{Weg75}]\label{lemma:hereditary}
Let $K$ be a $d$-collapsible complex on vertex set $V$, and let $U\subset V$. Then, $K[U]$ is also $d$-collapsible.
\end{lemma}

\begin{lemma}[Khmelnitsky \cite{khmel}, see also {\cite[Prop 3.7]{kim2019complexes}}]
\label{lemma:coll_of_link}
Let $K$ be a $d$-collapsible complex, and let $\sigma\in K$. Then, $\lk(K,\sigma)$ is also $d$-collapsible.
\end{lemma}

It will be convenient to use the following equivalent definition of $d$-collapsibility:

\begin{lemma}[Tancer {\cite[Lemma 5.2]{Tan10}}]\label{lemma:dcoll_eq_def_recursive}
Let $K$ be a simplicial complex. Then, $K$ is $d$-collapsible if and only if
one of the following holds:
\begin{itemize}
\item $\dim(K)<d$, or
\item There exists some $\sigma\in K$ such that $|\sigma|=d$, $\sigma$ is contained in a unique maximal face $\tau\neq \sigma$ of $K$, and $\cost(K,\sigma)$ is $d$-collapsible.
\end{itemize}
\end{lemma}

\section{Some topological preliminaries}\label{sec:tolerance_homological}

In this section we prove some auxiliary results on the homology groups of simplicial complexes that we will later need.

Using the Mayer-Vietoris exact sequence (Theorem~\ref{thm:mayer_vietoris}) and Leray's Nerve Theorem (Theorem~\ref{thm:nerve_theorem}), we can prove the following.

\begin{lemma}\label{lemma:nerve_contractible_intersections}
Let $X_1,\ldots, X_m$ be simplicial complexes, and let $X=\bigcup_{i=1}^m X_i$. If for all $I\subset[m]$ of size at least $2$, the complex $\cap_{i\in I} X_i$ is non-empty and acyclic, then
\[
    \fhomology{k}{X}\cong\bigoplus_{i=1}^m \fhomology{k}{X_i}.
\]
for all $k\geq -1$.
\end{lemma}
\begin{proof}
We argue by induction on $m$. For $m=1$ the claim is trivial. Assume $m>1$. Since $\bigcap_{i\in I} X_i$ is non-empty and acyclic for every $I\subset[m-1]$, we obtain, by the induction hypothesis,
\[
    \fhomology{k}{\bigcup_{i=1}^{m-1} X_i}\cong\bigoplus_{i=1}^{m-1}\fhomology{k}{X_i}
\]
for all $k\geq -1$. 

Since $X=\left(\bigcup_{i=1}^{m-1} X_i\right) \cup X_m$, we have by Theorem \ref{thm:mayer_vietoris} a long exact sequence
\[
\cdots \to \fhomology{k}{\bigcup_{i=1}^{m-1} (X_i\cap X_m)}\to \bigoplus_{i=1}^m \fhomology{k}{X_i} \to
\fhomology{k}{X} \to \fhomology{k-1}{\bigcup_{i=1}^{m-1} (X_i\cap X_m)}\to\cdots
\]
Hence, it is enough to show that 
\[
\fhomology{k}{\bigcup_{i=1}^{m-1} (X_i\cap X_m)}=0
\]
for all $k\geq -1$.

By the assumptions of this lemma, the nerve $N=N(\{X_i\cap X_m\}_{i=1}^{m-1})$ is the complete complex on vertex set $[m-1]$. Moreover, for all $I\subset [m-1]$, the complex
\[
    \bigcap_{i\in I} (X_i\cap X_m) = \bigcap_{i\in I\cup\{m\}} X_i
\]
is acyclic. Therefore, by Theorem \ref{thm:nerve_theorem}, we obtain
\[
\fhomology{k}{\bigcup_{i=1}^{m-1} (X_i\cap X_m)}\cong \fhomology{k}{N}=0
\]
for all $k\geq -1$. Thus, 
\[
    \fhomology{k}{X}\cong\bigoplus_{i=1}^m \fhomology{k}{X_i}
\]
for all $k\geq-1$.
\end{proof}

We will also need the following simple result about relative homology:

\begin{lemma}\label{lemma:relative_homology}
Let $X$ be a simplicial complex on vertex set $V$, and let $Y\subset X$ be a subcomplex. Assume that there is some $\sigma\in X$ and subcomplexes $W\subset Z\subset X[V\setminus \sigma]$ such that 
\[
    X\setminus Y = \{ \eta\cup \sigma: \eta\in Z\setminus W\}.
\]
Then,
\[
    \relhomology{k}{X,Y}\cong\relhomology{k-|\sigma|}{Z,W}
\]
for all $k$.
\end{lemma}
\begin{proof}
For all $k$, let $\phi_k: C_k(X,Y )\to C_{k-|\sigma|}(Z,W )$ be defined by
\[
    \phi_k(\eta\cup \sigma)= \eta
\]
and extended linearly. Note that the maps $\phi_k$ are linear isomorphisms. Denote by $\partial_k$ the boundary operator of $C_k(X,Y )$ and by $\partial'_k$ the boundary operator of $C_k(Z,W)$. 
We are left to show that $\phi$ is a chain map. That is, for any $\eta\in Z(k)\setminus W(k)$, we need to show that
\[
    \phi_{k+|\sigma|-1}(\partial_{k+|\sigma|}(\eta\cup \sigma))= \partial'_k( \phi_{k+|\sigma|}(\eta\cup \sigma)).
\]
Let $\eta=\{u_0,\ldots,u_k\}$. For any $i\in\{0,\ldots,k\}$, let $\eta_i=\{u_0,\ldots,u_{i-1},u_{i+1},\ldots,u_k\}$.
Then, since any subset of $\eta\cup\sigma$ belonging to $X\setminus Y$ must contain $\sigma$, we have
\[
    \partial_{k+|\sigma|}(\eta\cup \sigma) =
        \sum_{\substack{i\in\{0,\ldots,k\}:\\ \eta_i\cup\sigma\notin Y}}  (-1)^i \eta_i\cup\sigma = 
        \sum_{\substack{i\in\{0,\ldots,k\}:\\ \eta_i\notin W}}  (-1)^i \eta_i\cup\sigma.
\]
Hence,
\[
    \phi_{k+|\sigma|-1}(\partial_{k+|\sigma|}(\sigma\cup\eta))=  \sum_{\substack{i\in\{0,\ldots,k\}:\\ \eta_i\notin  W}}  (-1)^i \eta_i = \partial'_k(\eta) = \partial'_k(\phi_{k+|\sigma|}(\eta\cup\sigma)).
\]
So $C_k(X,Y )$ and $C_{k-|\sigma|}(Z,W )$ are isomorphic as chain complexes, and in particular have isomorphic homology groups.
\end{proof}

\section{Proof of Theorem \ref{thm:main_tolerance}}\label{sec:tolerance_main}

In this section, we present the proof of Theorem \ref{thm:main_tolerance}. 

Note that the construction of the tolerance complexes depends on the vertex set of the original complex. 
For the construction of tolerance complexes, we will consider the vertex set of $K[U]$ to be the set $U$, the vertex set of $\cost(K,\sigma)$ to be $V$, and the vertex set of $\lk(K,\sigma)$ to be $V\setminus \sigma$.

\begin{lemma}\label{lemma:relative_decomposition}
Let $K$ be a simplicial complex on vertex set $V$, and let $\sigma\in K$. Then,
\begin{align*}
    \tol{K}{t}\setminus& \tol{\cost(K,\sigma)}{t} \\= &\left\{\sigma\cup\eta:\, \eta\in \tol{\lk(K,\sigma)}{t}\setminus \left(\bigcup_{\substack{\sigma'\subset \sigma:\\ 1\leq |\sigma'|\leq t}} \tol{\lk(K[V\setminus\sigma'],\sigma\setminus\sigma')}{t-|\sigma'|} \right)\right\}.
\end{align*}
\end{lemma}
\begin{proof}
Suppose $\tau\in \tol{K}{t}\setminus \tol{\cost(K,\sigma)}{t}$.
Since $\tau\in \tol{K}{t}$, we can write $\tau=\tau'\cup\tau''$, where $\tau'\in K$ and $|\tau''|\leq t$. 
Moreover, we must have $\tau' \supset \sigma$. Otherwise, $\tau'\in \cost(K,\sigma)$, a contradiction to $\tau\notin \tol{\cost(K,\sigma)}{t}$.

Let $\eta=\tau\setminus \sigma$.
Then, we can write $\eta= (\tau'\setminus \sigma)\cup \tau''$. Since $\tau'\setminus\sigma\in \lk(K,\sigma)$, we obtain $\eta\in \tol{\lk(K,\sigma)}{t}$.
We claim that 
\[
\eta\notin \bigcup_{\substack{\sigma'\subset \sigma:\\ 1\leq |\sigma'|\leq t}} \tol{\lk(K[V\setminus\sigma'],\sigma\setminus\sigma')}{t-|\sigma'|}.
\]
Assume for contradiction that $\eta\in \tol{\lk(K[V\setminus \sigma'],\sigma\setminus \sigma')}{t-|\sigma'|}$ for some $\sigma'\subset \sigma$, $1\leq |\sigma'|\leq t$. Then, we can write
\[
    \eta=\eta_1\cup \eta_2,
\]
where $\eta_1\cap \sigma=\varnothing$, $\eta_1\cup(\sigma\setminus \sigma')\in K$ and $|\eta_2|\leq t-|\sigma'|$.
Hence, we obtain
\[
    \tau=\sigma\cup \eta= (\eta_1\cup (\sigma\setminus \sigma')) \cup (\sigma'\cup\eta_2).
\]
Since $\sigma\not\subset \eta_1\cup(\sigma\setminus\sigma')$ and $|\sigma' \cup \eta_2| \leq t$, we have $\tau\in\tol{\cost(K,\sigma)}{t}$, which is a contradiction to the assumption $\tau\in \tol{K}{t}\setminus \tol{\cost(K,\sigma)}{t}$.

For the opposite direction, suppose $\tau=\sigma\cup \eta$, where
\[
\eta\in \tol{\lk(K,\sigma)}{t}\setminus\left(\bigcup_{\substack{\sigma'\subset \sigma:\\ 1\leq |\sigma'|\leq t}} \tol{\lk(K[V\setminus\sigma'],\sigma\setminus\sigma')}{t-|\sigma'|}\right).
\]
We claim that $\tau \in \tol{K}{t}\setminus\tol{\cost(K,\sigma)}{t}$.
Since $\eta \in \tol{\lk(K,\sigma)}{t}$, we can write $\eta=\eta_1\cup \eta_2$, where $\eta_1\cap \sigma=\varnothing$, $\eta_1\cup \sigma\in K$ and $|\eta_2|\leq t$. 
Therefore, $\tau=(\eta_1\cup\sigma)\cup \eta_2\in \tol{K}{t}$. 
We are left to show that $\tau\notin \tol{\cost(K,\sigma)}{t}$. 
Assume for contradiction that $\tau\in \tol{\cost(K,\sigma)}{t}$. 
Then, we can write $\tau=\tau_1\cup \tau_2$, where $\tau_1\in K$, $\sigma\not\subset \tau_1$ and $|\tau_2|\leq t$. Let $\sigma'=\tau_2\cap \sigma$. 
Since $\sigma \not\subset \tau_1$ and $\sigma \subset \tau$, we must have $\sigma'\neq \varnothing$.
In particular, $1 \leq |\sigma'| \leq t$.
We can write $\eta$ as follows:
\[
\eta = \tau \setminus \sigma = (\tau_1\setminus (\sigma\setminus \sigma'))\cup (\tau_2\setminus \sigma').
\]
Note that $\tau_1\setminus (\sigma\setminus \sigma')\in \lk(K[V\setminus \sigma'],\sigma\setminus \sigma')$ and $|\tau_2\setminus \sigma'|\leq t-|\sigma'|$. Thus, $\eta\in \tol{\lk(K[V\setminus \sigma'],\sigma\setminus \sigma')}{t-|\sigma'|}$, which is a contradiction to the assumption on $\eta$. This completes the proof.
\end{proof}

By Lemma \ref{lemma:relative_homology}  and Lemma \ref{lemma:relative_decomposition}, we obtain:

\begin{corollary}\label{cor:relative_decomposition}
Let $K$ be a simplicial complex, and let $\sigma\in K$. Then, for all $k$, we have
\begin{multline*}
    \relhomology{k}{\tol{K}{t}, \tol{\cost(K,\sigma)}{t}}\\\cong \relhomology{k-|\sigma|}{ \tol{\lk(K,\sigma)}{t},\tol{\lk(K,\sigma)}{t}\cap \left( \bigcup_{\substack{\sigma'\subset \sigma:\\ 1\leq |\sigma'|\leq t}} \tol{\lk(K[V\setminus\sigma'],\sigma\setminus\sigma')}{t-|\sigma'|} \right)}.
\end{multline*}
\end{corollary}

\begin{proposition}
\label{prop:main}
Let $K$ be a simplicial complex, and $\sigma\in K$ such that $\sigma$ is contained in a unique maximal simplex $\sigma\cup U \in K$, where $U\neq \varnothing$. Then, for all $k$,
\begin{align*}
    &\relhomology{k}{\tol{K}{t},\tol{\cost(K,\sigma)}{t}} \\&\cong \bigoplus_{\substack{ W\subset V\setminus (\sigma\cup U):\\  |W|=t}}
    \fhomology{k-|\sigma|-1}{\bigcup_{\substack{\sigma'\subset\sigma:\\ 1\leq |\sigma'|\leq t}} \tol{\lk(K,\sigma\setminus \sigma')[U\cup W]}{t-|\sigma'|}}.
\end{align*}
\end{proposition}
\begin{proof}
Let
\[
    Y=\bigcup_{\substack{\sigma'\subset \sigma:\\ 1\leq |\sigma'|\leq t}} \tol{\lk(K[V\setminus\sigma'],\sigma\setminus\sigma')}{t-|\sigma'|}.
\]
By Corollary \ref{cor:relative_decomposition}, we have
\[
 \relhomology{k}{\tol{K}{t},\tol{\cost(K,\sigma)}{t}}\cong\relhomology{k-|\sigma|}{\tol{\lk(K,\sigma)}{t},\tol{\lk(K,\sigma)}{t}\cap Y}.
\]
By applying Theorem~\ref{thm:exact_sequence_pair} to the pair $\tol{\lk(K,\sigma)}{t}\cap Y \subset \tol{\lk(K,\sigma)}{t}$, we obtain a long exact sequence
\begin{multline*}
\cdots\to \fhomology{k-|\sigma|}{\tol{\lk(K,\sigma)}{t}}\to \relhomology{k-|\sigma|}{\tol{\lk(K,\sigma)}{t},\tol{\lk(K,\sigma)}{t}\cap Y}\to
\\
\to \fhomology{k-|\sigma|-1}{\tol{\lk(K,\sigma)}{t}\cap Y}\to
\fhomology{k-|\sigma|-1}{\tol{\lk(K,\sigma)}{t}}
\to\cdots
\end{multline*}
Note that $\lk(K,\sigma)=2^U$; therefore, 
\[
\tol{\lk(K,\sigma)}{t}=2^U\ast \{\tau\subset V\setminus(U\cup \sigma):\, |\tau|\leq t\}.
\]
In particular, since $U\neq \varnothing$, $\tol{\lk(K,\sigma)}{t}$ is contractible. Hence,
\[
\relhomology{k-|\sigma|}{\tol{\lk(K,\sigma)}{t},\tol{\lk(K,\sigma)}{t}\cap Y}\cong \fhomology{k-|\sigma|-1}{\tol{\lk(K,\sigma)}{t}\cap Y}.
\]
We can write
\begin{align}\label{eq:Y_W}
   \tol{\lk(K,\sigma)}{t}\cap Y =
   \bigcup_{\substack{W\subset V\setminus(\sigma\cup U):\\ |W|=t}} 2^{U\cup W} \cap Y
   =
   \bigcup_{\substack{W\subset V\setminus(\sigma\cup U):\\ |W|=t}} Y_W, 
\end{align}
where
\[
Y_W= Y[U\cup W] =\bigcup_{\substack{\sigma'\subset \sigma:\\ 1\leq |\sigma'|\leq t}} \tol{\lk(K,\sigma\setminus\sigma')[U\cup W]}{t-|\sigma'|}.
\] 
Let $m>1$, and let $W_1,\ldots,W_m\subset V\setminus(\sigma\cup U)$ be distinct sets, such that $|W_i|=t$ for all $i\in[m]$. Then,
\[
\bigcap_{i=1}^m Y_{W_i} = \bigcup_{\substack{\sigma'\subset \sigma:\\ 1\leq |\sigma'|\leq t}} \tol{\lk(K,\sigma\setminus\sigma')[U\cup (\cap_{i=1}^m W_i)]}{t-|\sigma'|}.
\]
Since $|\cap_{i=1}^m W_i|\leq t-1$, we have, for any $v\in \sigma$,
\[
    U\cup (\cap_{i=1}^m W_i) \in \tol{\lk(K,\sigma\setminus\{v\})[U\cup  (\cap_{i=1}^m W_i)]}{t-1}.
\]
In particular, it implies
\[
    U\cup (\cap_{i=1}^m W_i) \in\bigcap_{i=1}^m Y_{W_i},
\]
and hence, we conclude 
\[
\bigcap_{i=1}^m Y_{W_i}= 2^{U\cup(\cap_{i=1}^m W_i)}.
\]
Since $U\neq \varnothing$, the intersection $\bigcap_{i=1}^m Y_{W_i}$ is non-empty and acyclic.
Therefore, by applying Lemma~\ref{lemma:nerve_contractible_intersections} to \eqref{eq:Y_W}, we obtain
\begin{multline*}
\fhomology{k-|\sigma|-1}{\tol{\lk(K,\sigma)}{t}\cap Y}\cong \bigoplus_{\substack{ W\subset V\setminus (\sigma\cup U):\\  |W|=t}} \fhomology{k-|\sigma|-1}{Y_W}
\\
\cong\bigoplus_{\substack{ W\subset V\setminus (\sigma\cup U):\\  |W|=t}}\fhomology{k-|\sigma|-1}{\bigcup_{\substack{\sigma'\subset\sigma:\\ 1\leq |\sigma'|\leq t}} \tol{\lk(K,\sigma\setminus \sigma')[U\cup W]}{t-|\sigma'|}},
\end{multline*}
as wanted.
\end{proof}

Recall that $h(t,d)$ is defined as follows: $h(0,d)=d$ for all $d\geq 0$, and for $t>0$,
\[
    h(t,d)= \left(\sum_{s=1}^{\min\{t,d\}} \binom{d}{s}(h(t-s,d)+1)\right) + d.
\]

\begin{lemma}\label{lemma:h_function}
\begin{itemize}
 \item[]
    \item $h(t,1)=2t+1$,
    \item $h(1,d)=d^2+2d$, 
    \item For fixed $t$, $h(t,d)=O(d^{t+1})$.
\end{itemize}

\end{lemma}
\begin{proof}
First, we show that $h(t,1)=2t+1$. We argue by induction on $t$. For $t=0$ we have $h(0,1)=1=2t+1$. Now, assume $t>0$. Then, by the definition of $h(t,d)$ and the induction hypothesis, we obtain
\[
    h(t,1)= h(t-1,1)+1+1= 2(t-1)+3= 2t+1.
\]

Next, we show that $h(1,d)=d^2+2d$. 
Indeed, this follows immediately from the definition of $h(t,d)$
\[
    h(1,d)= d(h(0,d)+1)+d = d^2+2d.
\]

Finally, we show that, for fixed $t$, $h(t,d)=O(d^{t+1})$. We argue by induction on $t$. For $t=0$ we have $h(0,d)=d=O(d)$. Let $t>1$. We will show that there is some constant $C_t$ such that, for large enough $d$, $h(t,d)\leq C_t d^{t+1}$. By the definition of $h(t,d)$ and the induction hypothesis, we have,
\begin{align*}
       h(t,d)&= \left(\sum_{s=1}^{t} \binom{d}{s}(h(t-s,d)+1)\right) + d
       \\
       &\leq \left(\sum_{s=1}^t \frac{d^s}{s!} (C_{t-s} d^{t-s+1} +1)\right)+d
       \\
       &= \left( \sum_{s=1}^t\frac{C_{t-s}}{s!}\right) d^{t+1} + \left(\sum_{s=1}^t \frac{d^s}{s!} +d \right)
       \\
       &\leq C_t d^{t+1}
\end{align*}
for $C_t >\sum_{s=1}^t\frac{C_{t-s}}{s!} $ and large enough $d$. So, for fixed $t$, we have $h(t,d)=O(d^{t+1})$.
\end{proof}

Now we are ready to prove our main result.
\begingroup
\def\thetheorem{\ref{thm:main_tolerance}}
\begin{theorem}
Let $K$ be a $d$-collapsible complex, and let $t\geq 0$ be an integer. Then, $\tol{K}{t}$ is $h(t,d)$-Leray.
\end{theorem}
\addtocounter{theorem}{-1}
\endgroup
\begin{proof}
Let $V$ be the vertex set of $K$.
We will show that $\fhomology{k}{\tol{K}{t}}=0$ for $k\geq h(t,d)$.
This is sufficient to prove the statement of the theorem, since $\tol{K}{t}[W]=\tol{K[W]}{t}$ and, by Lemma \ref{lemma:hereditary}, $K[W]$ is $d$-collapsible for every $W\subset V$.

We argue by induction on $t$. If $t=0$ the statement obviously holds, since every $d$-collapsible complex is $d$-Leray.

Let $t\geq 1$. We argue by induction on the size of $K$, that is, the number of simplices in $K$.
If $\dim(K)<d$, then $\dim(\tol{K}{t})<d+t< h(t,d)$, so the statement holds.
Otherwise, by Lemma \ref{lemma:dcoll_eq_def_recursive}, there is some $\sigma\in K$ such that $|\sigma|=d$, $\sigma$ is contained in a unique maximal face $\tau\neq \sigma$ of $K$, and $\cost(K,\sigma)$ is $d$-collapsible.

Let $U=\tau\setminus \sigma\neq\varnothing$.
By applying Theorem \ref{thm:exact_sequence_pair} to the pair $\tol{\cost(K,\sigma)}{t} \subset \tol{K}{t}$, we obtain the following long exact sequence:
\[
 \cdots \to \fhomology{k}{\tol{\cost(K,\sigma)}{t}} \to \fhomology{k}{(\tol{K}{t})}\to \relhomology{k}{\tol{K}{t},\tol{\cost(K,\sigma)}{t}}\to \cdots
\]
By the induction hypothesis, $\fhomology{k}{\tol{\cost(K,\sigma)}{t}}=0$ for $k\geq h(t,d)$. Therefore, it is sufficient to show that $\relhomology{k}{\tol{K}{t},\tol{\cost(K,\sigma)}{t}}=0$ for $k\geq h(t,d)$.

By Proposition~\ref{prop:main}, it is sufficient to show that, for every $W\subset V\setminus(\sigma\cup U)$ of size $t$, the homology group
\[
\fhomology{k}{\bigcup_{\substack{\sigma'\subset\sigma:\\ 1\leq |\sigma'|\leq t}} \tol{\lk(K,\sigma\setminus \sigma')[U\cup W]}{t-|\sigma'|}}
\]
is trivial for $k\geq h(t,d)-|\sigma|-1=h(t,d)-d-1$. 
Note that, for any $\sigma'\subset\sigma$, by Lemma~\ref{lemma:hereditary} and Lemma~\ref{lemma:coll_of_link}, the complex $\lk(K,\sigma\setminus\sigma')[U\cup W]$ is $d$-collapsible. Hence, by the induction hypothesis, for any $\sigma'\subset\sigma$ of size $1\leq|\sigma'|\leq t$, the complex $\tol{\lk(K,\sigma\setminus\sigma')[U\cup W]}{t-|\sigma'|}$ is $h(t-|\sigma'|,d)$-Leray. So, by Theorem~\ref{thm:kalai_meshulam_leray},
\begin{align*}
L\left(\bigcup_{\substack{\sigma'\subset\sigma:\\ 1\leq |\sigma'|\leq t}} \tol{\lk(K,\sigma\setminus \sigma')[U\cup W]}{t-|\sigma'|}\right) &\leq \left(\sum_{\substack{\sigma'\subset\sigma:\\ 1 \leq |\sigma'|\leq t}} h(t-|\sigma'|,d)+1\right) -1 \\&= \left(\sum_{s=1}^{\min\{t,d\}}\binom{d}{s}(h(t-s,d)+1)\right)-1
\\&= h(t,d)-d-1.
\end{align*}
In particular, 
\[
\fhomology{k}{\bigcup_{\substack{\sigma'\subset\sigma:\\ 1\leq |\sigma'|\leq t}} \tol{\lk(K,\sigma\setminus \sigma')[U\cup W]}{t-|\sigma'|}}=0
\]
for $k\geq h(t,d)-d-1$, as wanted.
\end{proof}

\section[Improved bound for d=2, t=1]{Improved bound for $d=2, t=1$}\label{sec:tolerance_d2t1}
By Theorem \ref{thm:main_tolerance} and Lemma \ref{lemma:h_function}, we proved that the $1$-tolerance complex $\tol{K}{1}$ is $(d^2+2d)$-Leray for every $d$-collapsible complex $K$.
This is of the same order of magnitude, but larger than the conjectural bound $\eta(d+1,2)-1= \left\lfloor\left(\frac{d+3}{2}\right)^2\right\rfloor-1$ for $d > 1$.
In this section, we prove Theorem \ref{thm:main_tolerance_d2t1case}, which gives a tight bound for the Leray number of $\tol{K}{1}$, in the special case that $K$ is $2$-collapsible.

For the proof we will need the following Lemma:
\begin{lemma}\label{lemma:union_of_links_d2t1}
Let $K$ be a $2$-collapsible complex on vertex set $V$. Let $\sigma=\{u,v\}\in K$ such that $\sigma$ is contained in a unique maximal face $\sigma\cup U$, where $U\neq \varnothing$ . Let $w\in V\setminus(U\cup \sigma)$. Then,
\[
\fhomology{k}{\lk(K,v)[U\cup\{w\}] \cup \lk(K,u)[U\cup \{w\}]}=0
\]
for $k\geq 2$.
\end{lemma}
\begin{proof}
Let $A=\lk(K,v)[U\cup\{w\}]$ and $B=\lk(K,u)[U\cup\{w\}]$.
By Mayer-Vietoris (Theorem~\ref{thm:mayer_vietoris}), we have a long exact sequence
\[
\cdots\to 
\fhomology{k}{A}\bigoplus\fhomology{k}{B}
\to\fhomology{k}{A\cup B}\to  \fhomology{k-1}{A\cap B}\to \cdots
\]
Since $K$ is $2$-collapsible, then, by Lemma~\ref{lemma:hereditary} and Lemma~\ref{lemma:coll_of_link}, $A$ and $B$ are also $2$-collapsible. In particular, $\fhomology{k}{A}=\fhomology{k}{B}=0$ for $k\geq 2$. 
Therefore, it is enough to show that
\[
\tilde{H}_k(A\cap B)=0
\]
for $k\geq 1$.  
If $w\notin A\cap B$, then
\[
A\cap B=2^U,
\]
and the claim holds.
Otherwise, assume $w\in A\cap B$.
By Theorem \ref{thm:exact_sequence_link_costar}, we have a long exact sequence
\[
\cdots\to \rhomology{k}{(A\cap B)\setminus w}\to \rhomology{k}{A\cap B}\to \rhomology{k-1}{\lk(A\cap B,w)}\to\cdots
\]
Note that $(A\cap B)\setminus w = 2^U$; hence, $\fhomology{k}{(A\cap B)\setminus w}=0$ for all $k$.
Thus, it is enough to show that
\[
\fhomology{k}{\lk(A\cap B,w)}=
\fhomology{k}{\lk(K,\{v,w\})[U] \cap \lk(K,\{u,w\})[U]}=0
\]
for $k\geq 0$. 
Let 
\[
Z=\lk(K,\{v,w\})[U] \cap \lk(K,\{u,w\})[U].
\]
We will show that $Z$ is in fact a complete complex.

Note that a set $\tau\subset U$ is a missing face of $Z$ if and only if it is a missing face of $\lk(K,\{v,w\})[U]$ or a missing face of $\lk(K,\{u,w\})[U]$. Moreover, $\tau\subset U$ is a missing face of $\lk(K,\{v,w\})[U]$ if and only if there is some $\eta\subset\{v,w\}$ such that $\tau\cup \eta$ is a missing face of $K$. Similarly, $\tau$ is a missing face of $\lk(K,\{u,w\})$ if and only if there is some $\eta\subset\{u,w\}$ such that $\tau\cup \eta$ is a missing face of $K$.

Assume for contradiction that $Z$ contains a missing face $\tau\subset U$ of size at least two.
Recall that, since $K$ is $2$-collapsible, all the missing faces of $K$ are of size at most $3$.
Then, since $U\in\lk(K,\{u,v\})$, $\tau$ must be of the form $\tau = \{x,y\}$, where $\{x,y,w\}$ is a missing face of $K$.

Now, we look at the induced subcomplex $L=K[\{u,v,w,x,y\}]$. By Lemma \ref{lemma:hereditary}, $L$ is $2$-collapsible.
It is easy to check that the missing faces of $L$ are exactly the two sets $\{u,v,w\}$ and $\{x,y,w\}$. Then, we observe that $\lk(L, w)$ is a $1$-dimensional sphere, and that $L\setminus w= 2^{\{u,v,x,y\}}$ is contractible. Therefore, by applying Theorem \ref{thm:exact_sequence_link_costar}, we obtain $\tilde{H}_2(L)\neq 0$.
This implies that $L$ is not $2$-Leray, which is a contradiction to $L$ being $2$-collapsible.
Hence, $Z$ is a complete complex, and therefore $\fhomology{k}{Z}=0$ for all $k\geq 0$.
\end{proof}

\begin{remark}[\bf{Theorem \ref{thm:main_tolerance_d2t1case}}]
{\em
Let $K$ be a $2$-collapsible complex. Then, $\tol{K}{1}$ is $5$-Leray.
}
\end{remark}
\begin{proof}
The proof is exactly the same as the $t=1$ case of the proof of Theorem \ref{thm:main_tolerance},  except that we replace the use of the Kalai-Meshulam bound (Theorem \ref{thm:kalai_meshulam_leray}) by Lemma \ref{lemma:union_of_links_d2t1}.
\end{proof}

We close the section with the following conjecture, which is a weaker version of  Conjectures~\ref{conj:tolerance_coll_to_coll} and \ref{conj:tolerance_leray_to_leray}.
\begin{conjecture}
Let $K$ be a $d$-collapsible simplicial complex, and let $t \geq 0$.
Then, $\tol{K}{t}$ is $(\eta(d+1,t+1)-1)$-Leray.
\end{conjecture}
It would be interesting to settle this conjecture, at least in the special case $t=1$.

\section{Tolerance in colorful Helly theorems}\label{sec:colorful}
The colorful Helly theorem is one of the most important generalizations of Helly's theorem.
It was observed by Lov\'{a}sz, and first appeared in B\'{a}r\'{a}ny's paper \cite{Bar82}.
It asserts the following.
\begin{theorem}[Lov\'{a}sz, B\'{a}r\'{a}ny~\cite{Bar82}]\label{thm:col_hel}
Let $\mathcal{F}_1,\mathcal{F}_2,\ldots,\mathcal{F}_{d+1}$ be finite families of convex sets in $\mathbb{R}^d$, and let $\mathcal{F}$ be their disjoint union.
Suppose every subfamily of $\mathcal{F}$ that contains exactly one member from each $\mathcal{F}_i$ has a point in common.
Then, some $\mathcal{F}_i$ has a point in common.
\end{theorem}

Note that the colorful Helly theorem implies Helly's theorem, by assuming all $\mathcal{F}_i$'s are identical.
In \cite[Theorem 4.4]{MO11}, a tolerant version of the colorful Helly theorem in the plane was proved.  Here is a more general statement. 
\begin{theorem}
\label{thm:tol_col_hel}
Let $\mathcal{F}_1,\mathcal{F}_2,\ldots,\mathcal{F}_{d+1}$ be finite families of convex sets in $\mathbb{R}^d$, and let $\mathcal{F}$ be their disjoint union.
Suppose that every subfamily $\mathcal{F}'$ of $\mathcal{F}$ of size $\eta(d+1,t+1)$ has a subfamily $\mathcal{F}''\subset \mathcal{F}'$ of size $|\mathcal{F}''|\geq |\mathcal{F}'|-t$ such that any subfamily of $\mathcal{F}''$ containing exactly one member from each $\mathcal{F}_i$ has a point in common.
Then, some $\mathcal{F}_i$ has a point in common with tolerance $t$.
\end{theorem}
For completeness, we present a proof, closely following the argument presented in \cite{MO11} for the special case $d=2, t=1$.
\begin{proof}[Proof of Theorem \ref{thm:tol_col_hel}]

Consider the hypergraph $\mathcal{H}$ whose vertex set is $\mathcal{F}$ and whose edges are the subfamilies that contain exactly one member from each $\mathcal{F}_i$ and that do not have a point in common. 

By the assumption of the theorem, every subhypergraph of $\mathcal{H}$ consisting of $\eta(d+1,t+1)$ vertices has covering number at most $t$. Therefore, by the definition of $\eta(d+1,t+1)$, $\mathcal{H}$ has covering number at most $t$. That is, there is a subfamily $\mathcal{F}'$ of $\mathcal{F}$ of size $|\mathcal{F}'|\geq |\mathcal{F}|-t$ such that every subfamily of $\mathcal{F}'$ consisting of exactly one set from each $\mathcal{F}_i$ has a point in common. Therefore, by Theorem~\ref{thm:col_hel} (applied to the family $\mathcal{F}'$), there is some $i$ such that $\mathcal{F}_i\cap \mathcal{F}'$ has a point in common.
Since $|\mathcal{F}_i \cap \mathcal{F}'|\geq |\mathcal{F}_i|-t$, $\mathcal{F}_i$ has a point in common with tolerance $t$.
\end{proof}
Similarly as above, we observe that Theorem~\ref{thm:tol_col_hel} implies Theorem~\ref{thm:tolerance_helly_convex}, by assuming all $\mathcal{F}_i$'s are identical.

As an application of our main results, we obtain new tolerant variants of the colorful Helly theorem.

A family $M$ of subsets of a non-empty set $V$ is a {\em matroid} if it satisfies
\begin{itemize}
    \item[(i)] $\varnothing \in M$, 
    \item[(ii)] for all $A' \subset A \subset V$, if $A \in M$ then $A' \in M$, and
    \item[(iii)] if $A, B \in M$ and $|A| < |B|$, then there exists $x \in B \setminus A$ such that $A \cup \{x\} \in M$.
\end{itemize}
The {\em rank function} of a matroid $M$ on $V$ is a function $\rho: 2^V \to \mathbb{N}$ such that for every $W \subset V$, $\rho(W)$ equals to the maximal size of $W' \subset W$ with $W' \in M$.
Note that the conditions (i) and (ii) allow us to regard a matroid $M$ as a simplicial complex.
The colorful Helly theorem can be generalized topologically as follows:

\begin{theorem}[Kalai and Meshulam, {\cite[Theorem 1.6]{KM05}}]\label{thm:top_col_hel}
Let $K$ be a $d$-Leray complex on $V$ and let $M$ be a matroid on $V$ with rank function $\rho$.
If $M \subset K$, then there exists $\sigma \in K$ such that $\rho(V \setminus \sigma) \leq d$.
\end{theorem}

Taking $K$ to be the nerve of the family $\mathcal{F}$ and $M$ to be the matroid whose members are the subfamilies of $\mathcal{F}$ containing at most one member from each $\mathcal{F}_i$, we can recover Theorem \ref{thm:col_hel} from Theorem \ref{thm:top_col_hel}.

By combining Theorem~\ref{thm:top_col_hel} with Theorems~\ref{thm:main_tolerance} and \ref{thm:main_tolerance_d2t1case}, we obtain the following results.

\begin{theorem}\label{thm:tol_top_col_hel}
Let $K$ be a $d$-collapsible complex on $V$ and let $M$ be a matroid on $V$ with rank function $\rho$. If $M \subset \tol{K}{t}$, then there exists $\sigma \in \tol{K}{t}$ such that $\rho(V \setminus \sigma) \leq h(t,d)$.
\end{theorem}

\begin{theorem}\label{thm:tol_top_col_hel_d2t1}
Let $K$ be a $2$-collapsible complex on $V$ and let $M$ be a matroid on $V$ with rank function $\rho$. If $M \subset \tol{K}{1}$, then there exists $\sigma \in \tol{K}{1}$ such that $\rho(V \setminus \sigma) \leq 5$.
\end{theorem}

As a corollary of Theorem \ref{thm:tol_top_col_hel_d2t1}, we obtain the following version of the tolerant colorful Helly theorem in the plane:
\begin{corollary}\label{cor:plane_six}
Let $\mathcal{F}_1,\mathcal{F}_2,\ldots,\mathcal{F}_6$ be finite families of convex sets in the plane, and let $\mathcal{F}$ be their disjoint union.
Suppose that  every subfamily $\mathcal{F}'$ of $\mathcal{F}$ that contains exactly one set from each $\mathcal{F}_i$ has a point in common with tolerance $1$.
Then, some $\mathcal{F}_i$ has a point in common with tolerance $1$.
\end{corollary}

It may be interesting to try to find a direct combinatorial proof of Corollary~\ref{cor:plane_six}.


\bibliographystyle{abbrv}
\bibliography{biblio}

\end{document}